\documentclass{article}

\usepackage[T1]{fontenc}
\usepackage{graphicx}
\usepackage{amsmath,amssymb,amsthm}
\usepackage{hyperref}
\usepackage{color}
\usepackage{array}
\usepackage{url}

\usepackage{graphics}
\usepackage{tikz}
\usetikzlibrary{arrows,decorations.pathmorphing,backgrounds,positioning,fit,petri,chains, matrix,scopes}

\newtheorem{theorem}{Theorem}
\newtheorem{lemma}{Lemma}

\newtheorem{corollary}{Corollary}
\newtheorem{proposition}{Proposition}
\newtheorem{definition}{Definition}
\newtheorem{remark}{Remark}
\newtheorem{assumption}{Assumption}

\newcommand{\zTheta}[1] {\mathrm{\Theta}\left(#1\right)}
\newcommand{\BigO}[1] {O\left(#1\right)}
\newcommand{\smallO}[1] {o\left(#1\right)}

\newcommand{\cZ}{\mathcal{Z}}
\newcommand{\cA}{\mathcal{A}}
\newcommand{\cU}{\mathcal{U}}
\newcommand{\cL}{\mathcal{L}}
\newcommand{\cS}{\mathcal{S}}
\newcommand{\cK}{\mathcal{K}}
\newcommand{\cF}{\mathcal{F}}
\newcommand{\seq}{\textsc{Seq}}
\newcommand{\nti}{n\to\infty}

\title{On the number of lambda terms with prescribed size of their De Bruijn representation\footnote{This
work was partially supported by FWF grant SFB F50-03.}}

\author{
Bernhard Gittenberger \and Zbigniew Go{\l}\k{e}biewski
\thanks{Institute for Discrete
Mathematics and Geometry, Technische Universit\"at Wien, Wiedner Hauptstrasse 8-10/104,
A-1040 Wien, Austria. 
Supported by FWF grant SFB F50-03}
}

\begin{document}
\maketitle

\begin{abstract}
John Tromp introduced the so-called 'binary lambda calculus' as a way to encode lambda terms in terms of binary words. 
Later, Grygiel and Lescanne conjectured that the number of binary lambda terms with $m$ free
indices and of size $n$ (encoded as binary words of length $n$) is $\smallO{n^{-3/2} \tau^{-n}}$ for $\tau \approx 1.963448\ldots$.
We generalize the proposed notion of size and show that for several classes of lambda
terms, including binary lambda terms with $m$ free indices, the number of terms of size $n$ is 
$\zTheta{n^{-3/2} \rho^{-n}}$ with some class dependent constant $\rho$, which in particular
disproves the above mentioned conjecture. 
A way to obtain lower and upper bounds for the constant near the leading term is presented and 
numerical results for a few previously introduced classes of lambda terms are given.
\end{abstract}

\section{Introduction}
The objects of our interest are lambda terms which are a basic object of lambda calculus.
A lambda term is a formal expression which is described by the grammar 
$M \, ::= \, x \, | \, \lambda x . M \, | \, (M \, N)$ 
where $x$ is a variable, the operation $(M \, N)$ is called application, and using the quantifier $\lambda$ is called abstraction.
In a term of the form $\lambda x . M$ each occurrence of $x$ in $M$ is called a bound variable.
We say that a variable $x$ is free in a term $M$ if it is not in the scope of any abstraction.
A term with no free variables is called closed, otherwise open. 
Two terms are considered equivalent if they are identical up to
renaming of the variables, \emph{i.e.}, more formally speaking, they can be
transformed into each other by $\alpha$-conversion.  

In this paper we are interested in counting lambda terms whose size corresponds to their De
Bruijn representation (\emph{i.e.} nameless expressions in the sense of
\cite{deBruijn:IndMath:1972}).

\begin{definition}
A De Bruijn representation is a word described by the following specification:
\[
M \, ::= \, n \, | \, \lambda M \, | \, M \, N 
\]
where $n$ is a positive integer, called a De Bruijn index. 
Each occurrence of a De Bruijn index is called a variable and each
$\lambda$ an abstraction. A variable $n$ of a De Bruijn representation $w$ 
is bound if the prefix of $w$
which has this variable as its last symbol contains at least $n$ times the symbol $\lambda$,
otherwise it is free. The abstraction which binds a variable $n$ is the $n$th $\lambda$ before the
variable when parsing the De Bruijn representation from that variable $n$ backwards to the first symbol. 
\end{definition}

For the purpose of the analysis we will use the notation consistent with the one used in~\cite{DBLP:journals/corr/BendkowskiGLZ15}. 
This means that the variable $n$ will be represented as a sequence of $n$ symbols, namely as a
string of $n-1$ so-called 'successors' $S$ and a so-called 'zero' $0$ at the end. Obviously, there
is a one to one correspondence between equivalence classes of lambda terms (as described in the
first paragraph) and De Bruijn representations. 
For instance, the De Bruijn representation of the lambda-term 
$\lambda x . \lambda y . x y$ (which is e.g. equivalent to $\lambda a . \lambda b . a b$ or
$\lambda y . \lambda x . y x$) is $\lambda \lambda 2 1$; using the notation with successors this
becomes $\lambda \lambda ((S 0) 0)$.

In this paper we are interested in counting lambda terms of given size where we use a general
notion of size which covers several previously studied models from the literature. We count the
building blocks of lambda terms,  zeros, successors, abstractions and applications, 
with size $a, b, c$ and $d$, respectively. Formally, 
\[
|0|=a, \qquad | S n | = |n| + b, \qquad |\lambda M| = |M| + c, \qquad |M N| = |M| + |N| + d.
\]

Thus we have for the example given above
$| \lambda \lambda ((S 0) 0) | = 2 a + b + 2 c + d.$ 
Assigning sizes for the symbols like above covers several previously introduced notions of size:
\begin{itemize}
	\item so called 'natural counting' 
(introduced in~\cite{DBLP:journals/corr/BendkowskiGLZ15}) 
where $a = b = c = d = 1$,
	\item so called 'less natural counting' 
(introduced in~\cite{DBLP:journals/corr/BendkowskiGLZ15}) 
where $a = 0, b = c = 1, d = 2$.
 	\item binary lambda calculus 
(introduced in~\cite{tromp:DSP:2006:628}) 
where $b = 1, a = c = d = 2$,
\end{itemize}

\begin{assumption}\label{i:ass:1}
Throughout the paper we will make the following assumptions about the constants $a, b, c, d$:
\begin{enumerate}
	\item $a,b,c,d$ are nonnegative integers,
	\item $a+d \geq 1$,
	\item $b, c \geq 1$,
	\item $\gcd(b,c,a+d) = 1$.
\end{enumerate}
\end{assumption}

If the zeros and the applications both had size $0$ (\emph{i.e.} $a+d=0$), then we would have infinitely
many terms of the given size, because one can insert arbitrary many applications and zeros into a
term without increasing its size. If the successors or the abstractions had size $0$ (\emph{i.e.} $b$ or
$c$ equals to $0$), then we would again have infinitely many terms of given size, because one can
insert arbitrarily long strings of successors or abstractions into a term without increasing its
size. The last assumption is more technical in its nature. It ensures that the generating function
associated with the sequence of the number of lambda-terms will have exactly one singularity on
the circle of convergence. 

{\bf Notations.} 
We introduce some notations which will be frequently used throughout
the paper: If $p$ is a polynomial, then $\mathrm{RootOf}\left\{p\right\}$ will denote the
smallest positive root of $p$. Moreover, we will write $[z^n]f(z)$ for the $n$th coefficient of
the power series expansion of $f(z)$ at $z=0$ and $f(z) \prec g(z)$ (or $f(z) \preceq g(z)$) to
denote that $\left[z^n\right] f(z) < \left[z^n\right] g(z)$ (or $\left[z^n\right] f(z) \leq
\left[z^n\right] g(z)$) for all integers $n$. 

{\bf Plan of the paper.} 
The primary aim of this paper is the asymptotic enumeration of closed lambda terms of given size
with the size tending to infinity. In the next section we define several classes of lambda terms
as well as the generating function associated with them, present our main results and prove
several auxiliary results which will be important in the sequel. We derive the 
asymptotic equivalent of the number of closed terms of given size up to a constant factor. This is
established by construction of upper and lower bounds for the coefficients of the generating
functions. These constructions are done in Sections~\ref{sectionupper} and~\ref{sectionlower}. To
get fairly accurate numerical bounds we present a method for improving the previously obtained
bounds in Section~\ref{sectionimprovement}. Finally, Section~\ref{sectionprevious} is devoted to
the derivation of very accurate results for classes of lambda terms which have been previously studied in the literature.

\section{Main results}

In order to count lambda terms of a given size we set up a formal equation which is then
translated into a functional equation for generating functions. For this we will utilise the 
symbolic method developped in~\cite{Flajolet:2009:AC:1506267}. 

Let us introduce the following atomic classes: the class of zeros $\cZ$, the class of successors $\cS$, the class of abstractions $\cU$ and the class of applications $\cA$.
Then the class $\cL_{\infty}$ of lambda terms can be described as follows:
\begin{equation}\label{mr:eq:0}
\cL_{\infty} = \seq (\cS) \times \cZ + \cU  \times \cL_{\infty} + \cA \times \cL_{\infty}^2
\end{equation}
The number of lambda terms of size $n$, denoted by $L_{\infty,n}$, is $\left|\left\{t \in
\cL_{\infty}: |t|=n \right\}\right|$.
Let $L_{\infty}(z) = \sum_{n \geq 0} L_{\infty,n} z^n$ be the generating function associated with
$\cL_{\infty}$. 
Then specification~(\ref{mr:eq:0}) gives rise to a functional equation for the generating function $L_{\infty}(z)$:
\begin{equation}\label{mr:eq:1}
L_{\infty}(z) = z^a \sum_{j=0}^{\infty} z^{bj} + z^c L_{\infty}(z) + z^d L_{\infty}(z)^2.
\end{equation}
Solving \eqref{mr:eq:1} we get
\[
L_{\infty}(z) = \frac{1 - z^c - \sqrt{(1-z^c)^2 - \frac{4 z^{a+d}}{1-z^b}}}{2 z^d}, 
\]
which defines an analytic function in a neighbourhood of $z=0$. 

\begin{proposition}\label{i:prop:1}
Let $\rho = \mathrm{RootOf}\left\{(1-z^b)(1-z^c)^2 - 4 z^{a+d}\right\}$. Then
\begin{equation}\label{mr:eq:3}
L_{\infty}(z) = a_{\infty} + b_{\infty} \left(1 - \frac{z}{\rho}\right)^{\frac12} + \BigO{\left|1
- \frac{z}{\rho}\right|},
\end{equation}
for some constants $a_{\infty} > 0, b_{\infty} < 0$ that depend on $a,b,c,d$.
\end{proposition}
\begin{proof}
Let $f(z) = (1-z^b)(1-z^c)^2 - 4 z^{a+d}$. Then $\rho$ is the solution of $f(z) = 0$.
If we compute derivative $f'(z) = -4 (a + b) z^{a+b-1} - 2 c z^{c-1} (1 - z^b) (1 - z^c) - b
z^{b-1} (1 - z^c)^2$ we can observe that all three terms are negative for $0 < z < 1$.
Since $0 < \rho < 1$, the function $L_{\infty}(z)$ has an algebraic singularity of type $\frac12$
which means that its Newton-Puiseux expansion is of the form~(\ref{mr:eq:3}).

Since $L_{\infty}(z)$ is a power series with positive coefficients, we know that $a_{\infty} =
L_{\infty}(\rho) > 0$ and $b_{\infty} < 0$. 
\end{proof}

\begin{corollary}
The coefficients of $L_\infty(z)$ satisfy $[z^n]L_\infty(z)\sim C\rho^{-n}n^{-3/2}$, as $\nti$,
where $C=-b_\infty /(2\sqrt\pi)$. 
\end{corollary}

Let us define the class of $m$-open lambda terms, denoted $\cL_m$, as
\[
\cL_{m} = \left\{t \in \cL_{\infty}: \textrm{a prefix of at most } m \textrm{ abstractions } \lambda \textrm{ is needed to close the term}\right\}.
\]
The number of $m$-open lambda terms of size $n$ is denoted by $L_{m, n}$ and the generating
function associated with the class by $L_{m}(z) = \sum_{n \geq 0} L_{m,n} z^n$. 
Similarly to $\cL_{\infty}$, the class $\cL_{m}$ can be specified, and 
this specification yields the functional equation 
\begin{equation}\label{mr:eq:5}
L_{m}(z) = z^a \sum_{j=0}^{m-1} z^{bj} + z^c L_{m+1}(z) + z^d L_{m}(z)^2
\end{equation}
Note that $L_0(z)$ is the generating function of the set $\mathcal{L}_0$ of closed lambda terms.

Let $\cK_m = \cL_{\infty} \setminus \cL_{m}$ and $K_{m}(z) = L_{\infty}(z) - L_{m}(z)$. Then using
\eqref{mr:eq:1}~and~\eqref{mr:eq:5} we obtain 
\begin{equation}\label{mr:eq:6}
K_{m}(z) = z^a \sum_{j=m}^{\infty} z^{bj} + z^c K_{m+1}(z) + z^d K_{m}(z) L_{\infty}(z) + z^d K_{m}(z) L_{m}(z).
\end{equation}
which implies 
\begin{equation}\label{mr:eq:7}
K_{m}(z) = \frac{z^{a+bm}}{(1-z^b)(1 - z^d (L_{\infty}(z) + L_{m}(z)))} + \frac{z^c}{1 - z^d (L_{\infty}(z) + L_{m}(z))}{} K_{m+1}(z).
\end{equation}
Note that $K_m(z)$ as well as $L_m(z)$ define analytic functions in a neighbourhood of $z=0$. 

Let us state the main theorem of the paper:
\begin{theorem}\label{thm:1}
Let $\rho = \mathrm{RootOf}\left\{(1-z^b)(1-z^c)^2 - 4 z^{a+d}\right\}$. Then there exist positive
constants $\underline C$ and $\overline C$ (depending on $a, b, c, d$ and $m$) such that the number of $m$-open lambda terms of size
$n$ satisfies
\begin{equation}
\liminf_{n \to \infty} \frac{\left[z^n\right] L_{m}(z)}{\underline{C} n^{-\frac32} \rho^{-n}} \geq 1
\quad \textrm{and} \quad
\limsup_{n \to \infty} \frac{\left[z^n\right] L_{m}(z)}{\overline{C} n^{-\frac32} \rho^{-n}} \leq 1,
\end{equation}
\end{theorem}

\begin{remark}
In case of given $a, b, c, d$ and $m$ we can compute numerically such constants $\underline C$ and
$\overline C$. This will be done for some of the models mentioned in the introduction.
\end{remark} 

Before proving this theorem we will present the key ideas needed for our proof. We introduce the class $\cL_{m}^{(h)}$ of lambda terms in $\cL_{m}$ where the length of each 
string of successors is bounded by a constant integer $h$. As before, set $L_{m,n}^{(h)} =
\left|\left\{t \in \cL_{m}^{(h)}: |t|=n \right\}\right|$ and $L_{m}^{(h)}(z) = \sum_{n \geq 0}
L_{m,n}^{(h)} z^n$. Then $L_{m}^{(h)}(z)$ satisfies the functional equation
\begin{equation}\label{mr:eq:9}
L_{m}^{(h)}(z) = 
\begin{cases} 
z^a \sum_{j=0}^{m-1} z^{bj} + z^c L_{m+1}^{(h)}(z) + z^d L_{m}^{(h)}(z)^2 & \textrm{if } m < h,
\\
z^a \sum_{j=0}^{h-1} z^{bj} + z^c L_{h}^{(h)}(z) + z^d L_{h}^{(h)}(z)^2 & \textrm{if } m \geq h.
\end{cases}
\end{equation}
Notice that for $m \geq h$ we have a quadratic equation for $L_m^{(h)}(z)=L_{h}^{(h)}(z)$ 
that has the solution
\[
L_{h}^{(h)}(z) = \frac{1-z^c-\sqrt{(1-z^c)^2 - 4 z^{a+d} \frac{1-z^{b h}}{1-z^b}}}{2 z^d}.
\]
For $m < h$ we have a relation between $L_{m}^{(h)}(z)$ and $L_{m+1}^{(h)}(z)$ which gives rise to a
representation of $L_m^{(h)}(z)$ in terms of nested radicands
(\emph{cf.}~\cite{DBLP:conf/analco/BodiniGG11}) after all. Indeed, for $m < h$ we have
\begin{equation}\label{nestedrad}
L_{m}^{(h)}(z) = 
\frac{
1 - \sqrt{r_m(z) + 2z^c \sqrt{r_{m+1}(z) + 2z^c \sqrt{\cdots \sqrt{r_{h-1}(z) + 2z^c \sqrt{r_h(z)}}}}}
}
{2z^d}
\end{equation}
where
\begin{equation*}
r_{j}(z) = 
\begin{cases}
1 - 4z^{a+d} \frac{1 - z^{j b}}{1-z^b} - 2z^c & \textrm{if } m \leq j < h-1, \\
1 - 4z^{a+d} \frac{1-z^{(h-1)b}}{1-z^b} - 2z^c + 2z^{2c} & \textrm{if } j = h-1, \\
(1-z^c)^2 - 4 z^{a+d} \frac{1-z^{b h}}{1-z^b} & \textrm{if } j = h.
\end{cases}
\end{equation*}
One can check that the dominant singularity $\rho_{m}^{(h)}$ of $L_m^{(h)}(z)$ comes from the real
root of $r_h(z)$ which is closest to the origin and that it is of type $\frac12$ (\emph{i.e.} 
$L_m^{(h)}(z) = a_{m}^{(h)} + b_{m}^{(h)} \left(1 - \frac{z}{\rho^{(h)}}\right)^{\frac12} +
\BigO{\left|1 - \frac{z}{\rho^{(h)}}\right|}$ as $z \to \rho_{m}^{(h)}$ for some constants $a_{m}^{(h)}, b_{m}^{(h)}$ depending on $m$ and $h$).
Notice that for all $m,k \geq 0$ we have $\rho_{m}^{(h)} = \rho_{k}^{(h)}$ and that is why we can drop the lower index and write just $\rho^{(h)}$ instead of $\rho_{m}^{(h)}$.
Moreover, $\rho^{(h)} > \rho$ as well as $\lim_{h \to \infty} \rho^{(h)} = \rho$.

Let us begin with computing the radii of convergence of the functions $K_m(z)$ and $L_m(z)$. For
the case of binary lambda calculus Lemmas~\ref{l:01}~and~\ref{l:1} were already proven 
in~\cite{GLBinLT}. To extend those results to our more general setting, we will use different
techniques. 

\begin{lemma}\label{l:0}
For all $m \geq 0$ the radius of convergence of $K_{m}(z)$ equals $\rho$ (the radius of
convergence of $L_{\infty}(z)$).
\end{lemma}

\begin{proof}
Inspecting (\ref{mr:eq:7}) reveals that the key part is $\frac{1}{1 - z^d (L_{\infty}(z) +
L_{i}(z))}$. This is the generating function of a sequence of combinatorial structures associated
with the generating function $z^d (L_{\infty}(z) + L_{i}(z))$. 
One can check that we are not in the supercritical sequence schema case (\emph{i.e.} a singularity
of considered fraction does not come from the root of its denominator,
see~\cite[pp.~293]{Flajolet:2009:AC:1506267}) because $1 - \rho^d \left(L_{\infty}(\rho) +
L_i(\rho)\right) > 0$. This follows from 
\[
\rho^d \left(L_{\infty}(\rho) + L_i(\rho)\right) \leq 2 \rho^d L_{\infty}(\rho) = 1 - \rho^c < 1.
\]
The first inequality holds because $L_{\infty}(\rho) \geq L_i(\rho)$ for all $i\ge 0$ and the
second one because $\rho > 0$. 
Moreover, the radius of convergence of $L_i(z)$ is larger than or equal to the radius of convergence of $L_{\infty}(z)$ because $\cL_i \subseteq \cL_{\infty}$.
Therefore, for all $m \geq 0$ the radius of convergence of $K_{m}(z)$ equals $\rho$, the radius of convergence of $L_{\infty}(z)$.
\end{proof}

\begin{lemma}\label{l:01}
All the functions $L_{m}(z)$, $m \geq 0$, have the same radius of convergence.
\end{lemma}

\begin{proof}
Let $\rho_m$ denote the radius of convergence of the function $L_{m}(z)$.
From the definition of the function $L_{m}(z)$ it is known that for all $m \geq 0$ and for all $n$ we have $\left[z^n\right] L_{m}(z) \leq \left[z^n\right] L_{m+1}(z)$ and therefore $\rho_m \geq \rho_{m+1}$.
Moreover, from (\ref{mr:eq:5}) we know
\[
L_{m+1}(z) = - z^{a-c} \sum_{j=0}^{m-1} z^{bj} + z^{-c} L_{m}(z) - z^{d-c} L_{m}(z)^2.
\]
Notice that $\rho_m \leq 1$ because $\rho_m \leq \rho^{(h)} < 1$ for $h \geq m$.
Then due to the fact that $z^{-c} L_{m}(z) - z^{d-c} L_{m}(z)^2$ has radius of convergence bigger
or equal $\rho_m$, we have $\rho_{m} \leq \rho_{m+1}$.
\end{proof}

\begin{lemma}\label{l:1}
For all $m \geq 0$ the radius of convergence of $L_{m}(z)$ equals $\rho$.
\end{lemma}

\begin{proof}
Take $L_{m}^{(m)}$, defined in~\eqref{mr:eq:9}.
Recall that $\rho^{(m)}, \rho_m$ and $\rho$ denote the radii of convergence of $L_{m}^{(m)}(z),
L_{m}(z)$ and $L_{\infty}(z)$, respectively.
Notice that for all $m, n \geq 0$ we have $\left[z^n\right] L_{m}^{(m)}(z) \leq \left[z^n\right] L_{m}(z) \leq \left[z^n\right] L_{\infty}(z) $. 
Thus we also have $\rho^{(m)} \geq \rho_m \geq \rho$.
Since $\rho^{(m)}$ is the smallest positive solution of the equation $(1-z)^c - 4z^{a+d}
\frac{1-z^{bm}}{1-z^b} = 0$, we have $\lim_{m \to \infty} \rho^{(m)} = \rho$, therefore
Lemma~\ref{l:01} implies $\rho_{m} = \rho$.
\end{proof}

In the next sections we will present how to obtain an upper and a lower bound for
$\left[z^n\right] L_{m}(z)$. 
The idea is to construct auxiliary functions satisfying certain inequalities and to use them to
construct further ones until we have the desired bound. The procedure follows the flowchart
depicted in Fig.~\ref{diagram}.

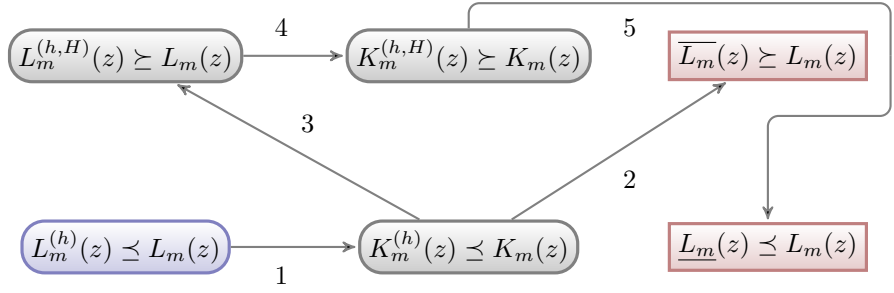
\begin{figure}[!h]
\begin{center}
\begin{tikzpicture}
[point/.style={coordinate},
>=stealth',thick,draw=black!50,
tip/.style={->,shorten >= 1pt},
every join/.style={rounded corners},
skip loop/.style={to path={-- ++(0,#1) -| (\tikztotarget)}},
fin/.style={
	rectangle,
	minimum size=6mm,
	very thick,
	draw=red!50!black!50,
	top color=white,
	bottom color=red!50!black!20,
	font=\ttfamily
},
aux/.style={
	rectangle, minimum size=6mm, rounded corners=3mm,
	very thick, draw=black!50,
	top color=white,
	bottom color=black!20,
	font=\ttfamily
},
start/.style={
	rectangle, minimum size=6mm, rounded corners=3mm,
	very thick, draw=blue!50!black!50,
	top color=white,
	bottom color=blue!50!black!20,
	font=\ttfamily
}
]

\matrix[column sep=3mm, row sep=3mm] {
\node (LhH) [aux] {$L_{m}^{(h,H)}(z) \succeq L	_{m}(z)$}; &
\node (p4) [label=above:$4$] {}; & 
\node (KhH) [aux] {$K_{m}^{(h,H)}(z) \succeq K_{m}(z)$}; & 
\node (p5) [label=above:$5$] {}; & 
\node (uL) [fin] {$\overline{L_m}(z) \succeq L_{m}(z)$}; & 
\node (x) [point] {}; \\
&
\node (p3) [label=above right:$3$] {}; & & 
\node (p2) [label=below:$2$] {}; & & \\
\node (Lh) [start] {$L_{m}^{(h)}(z) \preceq L_{m}(z)$}; & 
\node (p1) [label=below:$1$] {}; & 
\node (Kh) [aux] {$K_{m}^{(h)}(z) \preceq K_{m}(z)$}; & &
\node (lL) [fin] {$\underline{L_m}(z) \preceq L_{m}(z)$}; & \\
};

{ [start chain]
	\chainin (Lh);
	\chainin (Kh) [join=by tip];
	{ [start branch=uL]
		\chainin (uL) [join=by tip];
	}
	\chainin (LhH) [join=by tip];
	\chainin (KhH) [join=by tip];
	\chainin (x) [join=by {skip loop=7mm}];
	\chainin (lL) [join=by {skip loop=-8mm,tip}];
}
\end{tikzpicture}
\end{center}
\caption{The diagram illustrates the idea how we obtain the upper and the lower bound (in terms of the coefficients) for the function $L_m(z)$. Starting point is denoted by a blue node, the finish nodes are red.}\label{diagram}
\end{figure}

\section{Upper bound for $\left[z^n\right] L_{m}(z)$} \label{sectionupper}

Notice that for all integers $h$ and $m$ we have $\cL_{m}^{(h)} \subset \cL_{m}$.
Moreover, for all $m,h \geq 0$ there exists $n_{m}^{(h)}$ such that
$\left[z^n\right] L_{m}^{(h)}(z) = \left[z^n\right] L_{m}(z)$ if $n < n_{m}^{(h)}$ and
$\left[z^n\right] L_{m}^{(h)}(z) < \left[z^n\right] L_{m}(z)$ else. 
We will use those properties of $L_m^{(h)}(z)$ in order to derive a lower bound for 
the asymptotics of 
$\left[ z^n \right] K_{m}(z)$.

Note that \eqref{mr:eq:6} corresponds to an equation of the form 
$\cK_m = \cF\left(\cK_m, \cK_{m+1}, \cL_{\infty}, \cL_{m}\right)$. 
Now, define the new set $\cK_{m}^{(h)}:=\cF\left(\cK_{m}^{(h)}, \cK_{m+1}^{(h)}, \cL_{\infty}, \cL_{m}^{(h)}\right)$.
From the construction of $\cF$ and the properties of $\cL_{m}^{(h)}$ we know that $\cK_{m}^{(h)} \subseteq \cK_{m}$.
Let $K_{m,n}^{(h)} = \left|\left\{t \in \cK_{m}^{(h)}: |t|=n\right\}\right|$ and $K_{m}^{(h)}(z) = \sum_{n \geq 0} K_{m,n}^{(h)} z^n$.
Then $K_{m}^{(h)}(z)$ satisfies the functional equation
\begin{equation}\label{u:eq:7}
K_{m}^{(h)}(z) = z^a \sum_{j=m}^{\infty} z^{bj} + z^c K_{m+1}^{(h)}(z) + z^d K_{m}^{(h)}(z) L_{\infty}(z) + z^d K_{m}^{(h)}(z) L_{m}^{(h)}(z).
\end{equation}
In fact, what we did is that we replaced in the application operation every $m$-open lambda term 
(corresponding to the subterm $z^d K_{m}(z) L_{m}(z)$ of \eqref{mr:eq:6}) by an $m$-open lambda
term where each string of successors has bounded length (corresponding to $z^d K_{m}(z) L_{m}^{(h)}(z)$). Solving (\ref{u:eq:7}) we get
\begin{equation}\label{u:eq:8}
K_{m}^{(h)}(z) = \frac{z^{a-cm}}{1-z^b} \sum_{j=m}^{\infty} z^{j(b+c)} \prod_{i=m}^{j} \frac{1}{1
- z^d \left(L_{\infty}(z) + L_{i}^{(h)}(z)\right)}=: S_{m,\infty}(z).
\end{equation}

\begin{lemma}\label{l:2}
Let $\rho, a_{\infty}, b_{\infty}$ be as in Proposition~\ref{i:prop:1} and $\tilde{c}_i =
1/(1 - \rho^d (a_{\infty} + L_{i}^{(h)}(\rho)))$ and
$\tilde{d}_i = b_{\infty} \rho^d /(1 - \rho^d (a_{\infty} +
L_{i}^{(h)}(\rho)))^2$. 
Then $K_{m}^{(h)}(z)$ admits the expansion
\begin{equation}\label{u:eq:9}
K_m^{(h)}(z) = c_m^{(h)} + d_m^{(h)} \left(1-\frac{z}{\rho}\right)^{\frac12} +
\BigO{\left|1-\frac{z}{\rho}\right|}, \textrm { as } z\to\rho, 
\end{equation}
where
\begin{align*}
c_m^{(h)} &=
\begin{cases}
S_{m,h-1}(\rho) 
+ R_{c_m^{(h)}} & \textrm{if } m < h, 
\\ 
\frac{\rho^{a + b m}}{\left(1-\rho^b\right) \left(1 - \rho^{b+c} - \rho^d \left(a_{\infty} +
L_{h}^{(h)}(\rho)\right)\right)} & \textrm{else}, 
\end{cases} 
\\
d_m^{(h)} &= 
\begin{cases}
\frac{\rho^{a - c m}}{1-\rho^b}
\sum_{j=m}^{h-1} \rho^{j(b+c)} 
\sum_{i=m}^{j} \frac{\tilde{d}_i}{\tilde{c}_i} \prod_{k=m}^{j} \tilde{c}_{k} 
+ R_{d_m^{(h)}} & \textrm{if } m < h, \\
\frac{b_{\infty} \rho^{a + b m + d}}{\left(1-\rho^b\right) \left(1 - \rho^{b+c} - \rho^d
\left(a_{\infty} + L_{h}^{(h)}(\rho)\right)\right)^2} & \textrm{else}, 
\end{cases} 
\end{align*}
with 
\begin{align*}
R_{c_m^{(h)}} &=  \frac{\rho^{a + b h + c (h-m)}}{\left(1 - \rho^b\right) \left(1 - \rho^{b+c} - \rho^d \left(a_{\infty} + L_{h}^{(h)}(\rho)\right)\right)} \prod_{i=m}^{h-1} \tilde{c}_i, \\
R_{d_m^{(h)}} &= 
\frac{b_{\infty} \rho^{a + b h + c(h-m) + d}}{\left(1 - \rho^b\right) \left(1 - \rho^{b+c} - \rho^d \left(a_{\infty} + L_{h}^{(h)}(\rho)\right)\right)} 
\left(\prod_{i=m}^{h-1} \tilde{c}_i 
\right) \\
& \qquad \qquad \times
\left(
\sum_{i=m}^{h-1} \tilde{c}_i
+
\frac{1}{1 - \rho^{b+c} - \rho^d \left(a_{\infty} + L_{h}^{(h)}(\rho)\right)}
\right).
\end{align*}
\end{lemma}

\begin{proof}
Let us recall that for all $m \geq h$ we have $L_{m}^{(h)}(z)=L_h^{(h)}(z)$.
Therefore we can split the infinite sum $S_{m,\infty}(z)$ in \eqref{u:eq:8} into the finite one
$S_{m,h-1}(z)$ and the rest $S_{h,\infty}(z)$. 

{\bf Case I:} $m < h$. 
First, consider the finite sum $S_{m,h-1}(z)$. 
As in the proof of Lemma~\ref{l:0} 
we identify the key term, show that we are not in the supercritical case, and 
expand by means of Proposition~\ref{i:prop:1}. Eventually, this yields 
\[
\prod_{i=m}^{j} \frac{1}{1 - z^d (L_{\infty}(z) + L_{i}^{(h)}(z))} = \tilde{c}_{m,j} +
\tilde{d}_{m,j} \left(1-\frac{z}{\rho}\right)^{\frac12} + \BigO{\left|1-\frac{z}{\rho}\right|}
\]
where $\tilde{c}_{m,j} = \prod_{i=m}^{j} \tilde{c}_i$ and $\tilde{d}_{m,j} = \sum_{i=m}^{j}
\frac{\tilde{d}_i}{\tilde{c}_i} \prod_{k=m}^{j} \tilde{c}_{k}$. 
Since $\left(1-\frac{z}{\rho}\right)^{\frac12}$ does neither depend on $m$ nor on $j$ 
and $\frac{z^{a - cm}}{1 - z^b}$ has poles only on the unit circle, for all $m \geq 0$ we have  
\begin{equation*}
S_{m,h-1}(z)
= \tilde{\tilde{c}}_m + \tilde{\tilde{d}}_m
\left(1-\frac{z}{\rho}\right)^{\frac12} + \BigO{\left|1-\frac{z}{\rho}\right|}
\end{equation*}
where $\tilde{\tilde{c}}_m = \frac{\rho^{a - c m}}{1-\rho^b} \sum_{j=m}^{h-1} \rho^{j(b+c)}
\tilde{c}_{m,i}$ and $\tilde{\tilde{d}}_m = \frac{\rho^{a - c m}}{1-\rho^b} \sum_{j=m}^{h-1} \rho^{j(b+c)} \tilde{d}_{m,i}$.

Now, let us look on the infinite part of the sum in~\eqref{u:eq:8}. 
We will refer to its contributions to the first and second coefficient of the Newton-Puiseux
expansion \eqref{u:eq:9} as the remainders $R_{c_m^{(h)}}$ and $R_{d_m^{(h)}}$, respectively.
Since for all $m \geq h$ we have $L_{m}^{(h)}(z) = L_{h}^{(h)}(z)$, the sum can be rewritten as
\[
S_{h,\infty}(z)=
\left(\prod_{i=m}^{h-1} \frac{1}{1 - z^d \left(L_{\infty}(z) + L_{i}^{(h)}(z)\right)}\right) \cdot \frac{z^{a + b h + c (h-m)}}{\left(1-z^b\right) \left(1 - z^{b+c} - z^d \left(L_{\infty}(z) + L_{h}^{(h)}(z)\right)\right)}.
\]
We already know how to handle the product part of this expression, so let us consider the fraction $\frac{z^{a + b h + c (h-m)}}{\left(1-z^b\right) \left(1 - z^{b+c} - z^d \left(L_{\infty}(z) + L_{h}^{(h)}(z)\right)\right)}$.
Similarly to before, we have to check that the singularity of this function does not come from the
root of the denominator but from $L_{\infty}(z)$ (it cannot come from $L_{h}^{(h)}$ because it 
has a bigger radius of convergence than $L_{\infty}(z)$).
So, we have to show the inequality $1 - \rho^{b+c} - \rho^{d} \left(L_{\infty}(\rho) +
L_{h}^{(h)}(\rho)\right)>0$. But from $L_{\infty}(\rho) \geq L_{h}^{(h)}(\rho)$ and $0 <
\rho^{b+c} < \rho^{c} < 1$ we obtain 
$\rho^d (L_{\infty}(\rho) + L_{h}^{(h)}(\rho)) \leq 2 \rho^d L_{\infty}(\rho) = 1 - \rho^c < 1 -
\rho^{b+c}$
and hence the desired inequality indeed holds. Now, similarly to the previous case we use the 
Newton-Puiseux expansion of $L_{\infty}(z)$ at $\rho$ to derive an expansion of the infinite part
of the sum in \eqref{u:eq:8} and get the asserted result. 

{\bf Case II:} $m \geq h$. This case is easier, because the finite part of the sum in
\eqref{u:eq:8} does not exist and the other part can be evaluated to a closed form which can be
treated as above.
\end{proof}

Using the transfer lemmas of~\cite{DBLP:journals/siamdm/FlajoletO90} (applied to $K_m^{(h)}(z)$)
and $\left[z^n\right] K_{m}^{(h)}(z) \leq \left[z^n\right] K_m(z)$, 
we get $\liminf_{n \to \infty} \left(\left[z^n\right] K_m(z)\right) \cdot \Gamma(-1/2) n^{3/2} 
\rho^n / d_m^{(h)} \geq 1$.

\begin{corollary}\label{cor:1}
The number of $m$-open lambda terms of size $n$ satisfies
\[
\limsup_{n \to \infty} \frac{\left[z^n\right] L_{m}(z)}{\overline{C} n^{-\frac32} \rho^{-n}} \leq 1 \textrm{ where } \overline{C} = \frac{b_{\infty} - d_{m}^{(h)}}{\Gamma\left(-\frac12\right)}.
\]
\end{corollary}

\section{Lower bound for $\left[z^n\right] L_{m}(z)$} \label{sectionlower}

The idea behind obtaining a lower bound for $\left[z^n\right] L_{m}(z)$ is similar to
the one used for the upper bound. First we will find an upper bound for $\left[z^n\right]
K_{m}(z)$ using the function
\begin{equation}\label{l:eq:0}
L_m^{(h,H)}(z) = \sum_{n \geq 0} L_{m,n}^{(h,H)} z^n =  
\begin{cases}
L_{\infty}(z) - K_m^{(h)}(z) & \textrm{if } m < H, \\
L_{\infty}(z) & \textrm{else.}
\end{cases}
\end{equation}
Notice that for all $m,h,H,n \geq 0$ we have $\left[z^n\right] L_{m}(z) \leq \left[z^n\right] L_{m}^{(h,H)}$.
Let $\cL_{m}^{(h,H)}$ denote the class of combinatorial structures associated with
$L_m^{(h,H)}(z)$ and define the new set $\cK_m^{(h,H)}:=\cF\left(\cK_m^{(h,H)}, \cK_{m+1}^{(h,H)}, \cL_{\infty}, \cL_{m}^{(h,H)}\right)$.
From the construction of $\cF$ and the above properties of $\cL_{m}^{(h,H)}$ we know that $\cK_m \subseteq \cK_m^{(h,H)}$.
Let $K_{m,n}^{(h,H)} = \left|\left\{t \in \cK_{m}^{(h,H)}: |t|=n \right\}\right|$ and $K_{m}^{(h,H)}(z) = \sum_{n \geq 0} K_{m,n}^{(h,H)} z^n$. Then $K_{m}^{(h,H)}(z)$ is given by
\[
K_{m}^{(h,H)}(z) = z^a \sum_{j=m}^{\infty} z^{bj} + z^c K_{m+1}^{(h,H)}(z) + z^d K_{m}^{(h,H)}(z) L_{\infty}(z) + z^d K_{m}^{(h,H)}(z) L_{m}^{(h,H)}(z).
\]
Solving this equation and using \eqref{l:eq:0} we get
\begin{align}
K_{m}^{(h,H)}(z) &= 
\frac{z^{a-cm}}{1-z^b} \sum_{j=m}^{H-1} z^{j(b+c)} \prod_{i=m}^{j} \frac{1}{1 - z^d \left(L_{\infty}(z) + L_{m}^{(h,H)}(z)\right)}
\nonumber \\
&=
\frac{z^{a-cm}}{1-z^b} \sum_{j=m}^{H-1} z^{j(b+c)} \prod_{i=m}^{j} \frac{1}{1 - z^d \left(2 L_{\infty}(z) - K_i^{(h)}(z)\right)} + \nonumber \\
&
\frac{z^{a + bH + c (H-m)}}{\left(1-z^d\right) (1-z^{b+c}-2z^d L_{\infty}(z))}
\left(\prod_{i=m}^{H-1} \frac{1}{1 - z^d \left(2 L_{\infty}(z) - K_i^{(h)}(z)\right)} \right).
\label{l:eq:2}
\end{align}

\begin{lemma}\label{l:3}
Let $\rho$ be the radius of convergence of the function $L_{\infty}(z)$. Then the generating function $K_m^{(h,H)}(z)$ admits the following expansion
\begin{equation}\label{l:eq:3}
K_{m}^{(h,H)}(z) = c_m^{(h,H)} + d_m^{(h,H)} \left(1 - \frac{z}{\rho}\right)^{\frac12} + \BigO{\left|1 - \frac{z}{\rho}\right|},
\end{equation}
where
\begin{align*}
c_m^{(h,H)} &= 
\frac{\rho^{a - c m}}{1-\rho^b} \sum_{j=m}^{H-1} \rho^{j(b+c)} \prod_{i=m}^{j} \frac{1}{1 -
\rho^d \left(2 a_{\infty} - c_i^{(h)}\right)} 
+ R_{c_m^{(h,H)}},
\\
d_m^{(h,H)} &= \frac{\rho^{a - c m}}{1-\rho^b} \sum_{j=m}^{H-1} \rho^{j(b+c)} \sum_{i=m}^{j} \frac{\rho^d \left(2 b_{\infty} - d_i^{(h)}\right)}{1 - \rho^d \left(2 a_{\infty} - c_i^{(h)}\right)} \prod_{i=m}^{j} \frac{1}{1 - \rho^d \left(2 a_{\infty} - c_i^{(h)}\right)} + R_{d_m^{(h,H)}}, 
\end{align*}
$a_{\infty}, b_{\infty}$ and $c_i^{(h)}, d_i^{(h)}$ come from the expansion of $L_{\infty}(z)$ and
$K_i^{(h)}(z)$, respectively, at $\rho$ (see Proposition~\ref{i:prop:1} and the proof of
Lemma~\ref{l:2}) and
\begin{align*}
R_{c_m^{(h,H)}} &= \frac{\rho^{a + b H + c (H-m)}}{\left(1-\rho^d\right) \left(1-\rho^{b+c}-2\rho^d a_{\infty}\right)}
\left(\prod_{i=m}^{H-1} \frac{1}{1 - \rho^d \left(2 a_{\infty} - c_i^{(h)}\right)} \right), \\
R_{d_m^{(h,H)}} &= \frac{\rho^{a+bH+c(H-m)+d}}{\left(1-\rho^d\right) \left(1-\rho^{b+c}-2\rho^d a_{\infty}\right)} 
\left( \prod_{i=m}^{H-1} \frac{1}{1 - \rho^d \left(2 a_{\infty} - c_i^{(h)}\right)}\right) \\
& \qquad \qquad\times
\left(\frac{2 b_{\infty}}{1-\rho^{b+c}-2\rho^d a_{\infty}} + \sum_{i=m}^{H-1} \frac{\left(2 b_{\infty} - d_{i}^{(h)}\right)}{1 - \rho^d \left(2 a_{\infty} - c_i^{(h)}\right)}\right).
\end{align*}
\end{lemma}

As in the previous section, we apply the the transfer lemmas
of~\cite{DBLP:journals/siamdm/FlajoletO90} to $K_{m}^{(h,H)}(z)$ and use $\left[z^n\right]
K_{m}^{(h,H)}(z) \geq \left[z^n\right] K_{m}(z)$ to arrive at the following result: 

\begin{corollary}\label{cor:2}
The number of $m$-open lambda terms of size $n$ satisfies
\[
\liminf_{n \to \infty} \frac{\left[z^n\right] L_{m}(z)}{\underline{C} n^{-\frac32} \rho^{-n}} \geq
1\, \textrm{where } \underline{C} = \frac{b_{\infty} -
d_{m}^{(h,H)}}{\Gamma\left(-\frac12\right)}.
\]
\end{corollary}

\section{Improvement of the bounds}\label{sec:imp} \label{sectionimprovement}

The bounding functions $\underline{L_m}(z)=  L_\infty(z)-K_m^{(h,H)}(z)$ and $\overline{L_m}(z)=
L_\infty(z)-K_m^{(h)}(z)$ derived in the previous sections can be used in a straightforward way to 
compute numerical values for $\underline{C}$ and $\overline{C}$, given concrete values for 
$a,b,c,d$.
If we choose $h$ and $H$ big enough, then this will give us proper bounds. But in
practice they still leave a large gap, at least for values $h$ and $H$ that allow us to perform
the computation within a few hours on a standard PC.
For instance, in case of the natural counting for $h = H = 15$ we get $\underline{C}^{(nat)}
\approx 0.00404525\ldots$ and $\overline{C}^{(nat)} \approx 0.18086721\ldots$.

We show a simple way to improve $\underline{C}$ and $\overline{C}$.
Let us introduce the functions $L_{m,M}^{(h)}(z)$ and $L_{m,M}^{(h,H)}(z)$, defined by the
following equations:
\begin{align*}
L_{m,M}^{(h)}(z) &= 
\begin{cases}
z^a \sum_{j=0}^{m-1} z^{bj} + z^c L_{m+1,M}^{(h)}(z) + z^d L_{m,M}^{(h)}(z)^2
& \textrm{ if } m < M, \\
L_{\infty}(z) - K_{M}^{(h)}(z) & \textrm{ if } m=M, 
\end{cases}
\\
L_{m,M}^{(h,H)}(z) &=
\begin{cases}
z^a \sum_{j=0}^{m-1} z^{bj} + z^c L_{m+1,M}^{(h,H)}(z) + z^d L_{m,M}^{(h,H)}(z)^2
& \textrm{ if } m < M, \\
L_{\infty}(z) - K_{M}^{(h,H)}(z) & \textrm{ if } m=M.
\end{cases}
\end{align*}
These two function admit a representation in terms of nested radicals which is similar to
\eqref{nestedrad}: 
\begin{align*}
&L_{m,M}^{(h)}(z) = 
\\
& \quad \frac1{2z^d}\cdot \scriptstyle \left(1 - \sqrt{1 - 4z^{a+d}
\frac{1-z^{mb}}{1-z^b} - 2z^c + 2z^c \sqrt{\cdots \sqrt{1 - 4z^{a+d} \frac{1-z^{(M-1)b}}{1-z^b} -
2z^c + 2z^c \sqrt{L_{\infty}(z) - K_{M}^{(h)}}(z)}}}\right)\displaystyle,\\
&L_{m,M}^{(h,H)}(z) = 
\\
& \quad \frac1{2z^d}\cdot 
\scriptstyle \left(1 - \sqrt{1 - 4z^{a+d} \frac{1-z^{mb}}{1-z^b} - 2z^c + 2z^c \sqrt{\cdots
\sqrt{1 - 4z^{a+d} \frac{1-z^{(M-1)b}}{1-z^b} - 2z^c + 2z^c \sqrt{L_{\infty}(z) -
K_{M}^{(h,H)}(z)}}}}\right)\displaystyle.
\end{align*}

So, in $L_{m,M}^{(h)}(z)$ and $L_{m,M}^{(h,H)}(z)$ we are using exact expressions for $L_{m}(z)$
up to some constant $M$ and then replace $L_{M}(z)$ by a function that is its upper and lower
bound, respectively. Numerical results for some previously studied notions of size (see
Sections~\ref{sec:nat} and~\ref{sec:bin}) reveal a significant improvement in closing the gap
between the constants $\overline{C}, \underline{C}$ obtained by utilizing the functions
$L_{m,M}^{(h)}(z), L_{m,M}^{(h,H)}(z)$.

\section{Results for some previously introduced notions of size}
\label{sectionprevious}

\subsection{Natural counting}\label{sec:nat}

\begin{lemma}\label{l:5}
The following bounds hold
\begin{equation}
\liminf_{n \to \infty} \frac{\left[z^n\right] L_{0}^{(nat)}(z)}{\underline{C}^{(nat)} n^{-\frac32} \rho^{-n}} \geq 1
\quad \textrm{and} \quad
\limsup_{n \to \infty} \frac{\left[z^n\right] L_{0}^{(nat)}(z)}{\overline{C}^{(nat)} n^{-\frac32} \rho^{-n}} \leq 1\,
\end{equation}
where $\rho = \mathrm{RootOf}\{-1+3x+x^2+x^3\} \approx 0.295598...$ and $\underline{C}^{(nat)},
\overline{C}^{(nat)}$ are computable constants with numerical values $\underline{C}^{(nat)}
\approx 0.00404525\ldots$ and $\overline{C}^{(nat)} \approx 0.18086721\ldots$.
\end{lemma}

\begin{proof}
For 'natural counting' we have $a = b = c = d = 1$.
From Theorem~\ref{thm:1} we know that the radius of convergence of $L_{0}^{(nat)}(z)$ equals $\rho
= \mathrm{RootOf}\left\{-1 + 3 x + x^2 + x^3\right\} \approx 0.295598\ldots$.
We can also easily get $L_{\infty}(z)\sim a_{\infty}+b_{\infty} \sqrt{1-z/\rho}$ with 
\[
a_{\infty} = \frac{1-\rho}{2\rho} \approx 1.19149\ldots \textrm{ and } 
b_{\infty} = \frac{1}{\rho-1} \sqrt{\frac{1+\rho+\rho^2-\rho^3}{2\rho}} \approx 2.15093\ldots,  
\]
and from the transfer lemmas of~\cite{DBLP:journals/siamdm/FlajoletO90} we obtain 
$\left[z^n\right] L_{\infty}(z) \sim b_{\infty}n^{-3/2}\rho^{-n}/\Gamma(-1/2) 
\simeq (0.606767\ldots) \cdot n^{-\frac32} (3.38298\ldots)^{n}$, as $n \to \infty$.

\begin{table}[h]
\begin{center}
\begin{tabular}{| >{$} c <{$} | >{$} c <{$} | >{$} c <{$} | >{$} c <{$} | >{$} c <{$} |} \hline
h,H & c_0^{(h)} & d_0^{(h)} & c_0^{(h,H)} & d_0^{(h,H)} \\ \hline
1 & 0.855448 & -1.153959 & 1.086200 & -3.803686 \\ \hline
2 & 0.898032 & -1.313246 & 0.979519 & -2.581823 \\ \hline
3 & 0.917305 & -1.397536 & 0.958215 & -2.324953 \\ \hline
4 & 0.927248 & -1.444672 & 0.950295 & -2.236290 \\ \hline
5 & 0.932849 & -1.472308 & 0.946185 & -2.192353 \\ \hline
6 & 0.936128 & -1.488826 & 0.943824 & -2.167379 \\ \hline
7 & 0.938055 & -1.498647 & 0.942443 & -2.152790 \\ \hline \hline
8 & 0.939174 & -1.504385 & 0.941643 & -2.144335 \\ \hline
9 & 0.939813 & -1.507673 & 0.941187 & -2.139511 \\ \hline
10 & 0.940172 & -1.509525 & 0.940931 & -2.136799 \\ \hline
11 & 0.940372 & -1.510556 & 0.940788 & -2.135291 \\ \hline
12 & 0.940482 & -1.511125 & 0.940710 & -2.134460 \\ \hline
13 & 0.940543 & -1.511438 & 0.940667 & -2.134004 \\ \hline
14 & 0.940576 & -1.511608 & 0.940643 & -2.133755 \\ \hline
15 & 0.940594 & -1.511701 & 0.940630 & -2.133619 \\ \hline
\end{tabular}
\end{center}
\caption{Numbers rounded up to $7$ digits.}\label{n:tab:1}
\end{table}

Since we are most interested in the enumeration of closed lambda terms, we examine the
multiplicative constants in the leading term of the asymptotical lower and upper bound for $\left[z^n\right] L_0(z)$. From the formulas in Lemmas~\ref{l:2} and~\ref{l:3} we have computed the values for $c_0^{(h)}, d_0^{(h)}$ and $c_0^{(h,H)}, d_0^{(h,H)}$ for different constants $h$ and $H$ (see Table~\ref{n:tab:1}).

\begin{figure}[ht]
\centering
\includegraphics[width=0.7\textwidth]{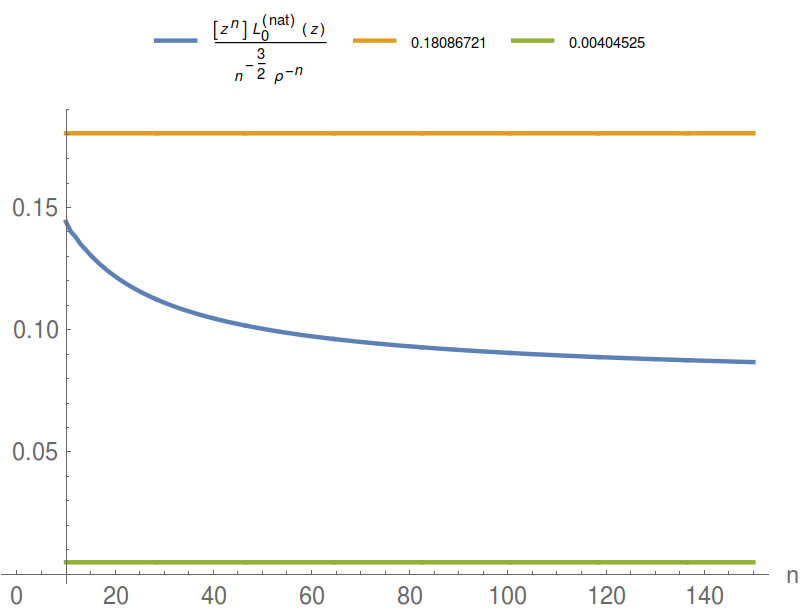}
\caption{Plot of the exact value of $\frac{\left[z^n\right] L_{0}^{(nat)}(z)}{n^{-\frac32} \rho^{-n}}$ and computed lower and upper bound for $h = H = 15$ and $10 \leq n \leq 150$.}\label{fig:1}
\end{figure}

As expected, the bigger $h$ and $H$ are, the more accurate is the bound we get (in this case for $\left[z^n\right] K_0(z)$). 
Taking the values of $d_0^{(h,H)}$ and $d_0^{(h)}$ for $h = H = 15$, Corollaries~\ref{cor:2}
and~\ref{cor:1} yield $\underline{C}^{(nat)} \approx 0.00404525\ldots$ and $\overline{C}^{(nat)} \approx 0.18086721\ldots$.
Notice that using the values of $d_0^{(h,H)}$ to compute $\underline{C}^{(nat)}$ gives 
non-trivial values only for $h > 7$ (for $1 \leq h \leq 7$ we get
negative numbers because in this case we have $\left|d_0^{(h,H)}\right| > 
\left|b_{\infty}\right|$).
\end{proof}

Figure~\ref{fig:1} illustrates the bounds we obtained and the exact values of the coefficients 
$\left[z^n\right] L_{0}^{(nat)}(z)$ for $10 \leq n \leq 150$.
Applying the approach discussed in Section~\ref{sec:imp} for $h = H = M =13$ we get the following
improvement.

\begin{lemma}\label{l:7}
The following bounds hold
\begin{equation}
\liminf_{n \to \infty} \frac{\left[z^n\right] L_{0}^{(nat)}(z)}{\underline{\underline{C}}^{(nat)} n^{-\frac32} \rho^{-n}} \geq 1
\quad \textrm{and} \quad
\limsup_{n \to \infty} \frac{\left[z^n\right] L_{0}^{(nat)}(z)}{\overline{\overline{C}}^{(nat)} n^{-\frac32} \rho^{-n}} \leq 1\,
\end{equation}
where $\rho = \mathrm{RootOf}\{-1+3x+x^2+x^3\} \approx 0.295598\ldots$ and
$\underline{\underline{C}}^{(nat)}, \overline{\overline{C}}^{(nat)}$ are computable constants with
numerical values $\underline{\underline{C}}^{(nat)} \approx 0.07790995266\ldots$ and $\overline{\overline{C}}^{(nat)} \approx 0.07790998229\ldots$.
\end{lemma}

\begin{figure}[ht]
\centering
\includegraphics[width=0.7\textwidth]{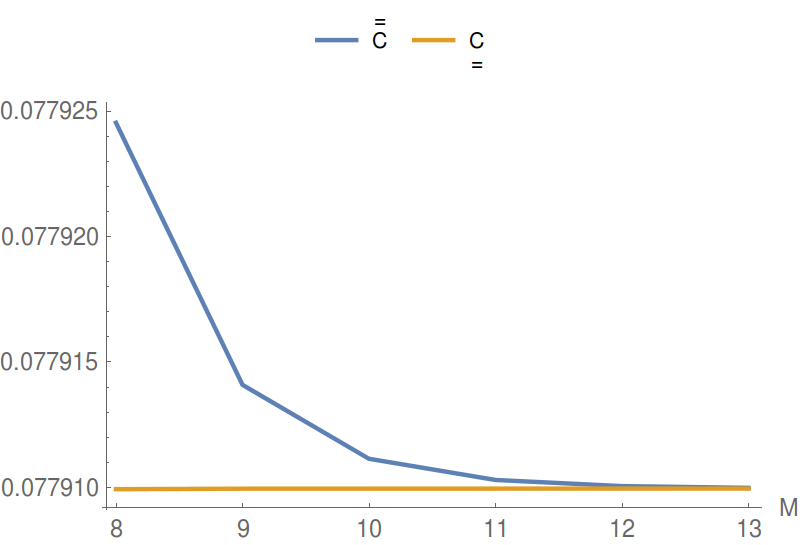}
\caption{Plot of the numerical values for the constants $\underline{\underline{C}}^{(nat)}, \overline{\overline{C}}^{(nat)}$ for $8 \leq h = H = M \leq 13$.}\label{fig:2}
\end{figure}

Figure~\ref{fig:2} illustrates how the improvement discussed in Section~\ref{sec:imp} allows to
reduce the gap between the constants $\underline{\underline{C}}^{(nat)},
\overline{\overline{C}}^{(nat)}$ for the lower and the upper bound.

\subsection{Binary lambda calculus}\label{sec:bin}

\begin{lemma}\label{l:8}
The following bounds hold
\begin{equation}
\liminf_{n \to \infty} \frac{\left[z^n\right] L_{0}^{(bin)}(z)}{\underline{\underline{C}}^{(bin)} n^{-\frac32} \rho^{-n}} \geq 1
\quad \textrm{and} \quad
\limsup_{n \to \infty} \frac{\left[z^n\right] L_{0}^{(bin)}(z)}{\overline{\overline{C}}^{(bin)} n^{-\frac32} \rho^{-n}} \leq 1\,
\end{equation}
where $\rho = \mathrm{RootOf}\{-1+x+2x^2-2x^3+3x^4+x^5\} \approx 0.509308\ldots$ and
$\underline{\underline{C}}^{(bin)}, \overline{\overline{C}}^{(bin)}$ are computable constants with
numerical values $\underline{\underline{C}}^{(bin)} \approx 0.01252417\ldots$ and $\overline{\overline{C}}^{(bin)} \approx 0.01254593\ldots$.
\end{lemma}

In order to proof this lemma it is enough to recall that in case of binary lambda calculus the
size defining constants are $b = 1$ and $a = c = d = 2$. Then we used the functions
$L_{0,13}^{13,13}(z), L_{0,13}^{13,13}(z)$ to obtain the numerical constants stated in
Lemma~\ref{l:8}.

\bibliography{bibliography}
\bibliographystyle{plain}

\newpage 

\section*{Appendix}

Below we present the proof of Lemma~\ref{l:3}.

\begin{proof}
Throughout the proof we will make use of Equation~(\ref{l:eq:2}).
Let us first focus on the finite sum: $\frac{z^{a-cm}}{1-z^b} \sum_{j=m}^{H-1} z^{j(b+c)} \prod_{i=m}^{j} \frac{1}{1 - z^d \left(2 L_{\infty}(z) - K_i^{(h)}(z)\right)}$.
The key part of it $\frac{1}{1 - z^d \left(2 L_{\infty}(z) - K_i^{(h)}(z)\right)}$ is the generating function of a sequence of structures enumerated by the denominator $z^d \left(2 L_{\infty}(z) -K_i^{(h)}(z)\right)$.
Like in case of Equation~(\ref{u:eq:8}), one can check that we are not in the supercritical sequence schema case because
\begin{align*}
1 - \rho^d \left(2 L_{\infty}(\rho) - K_i^{(h)}(\rho)\right) & > 0 \\
\rho^d \left(2 L_{\infty}(\rho) - K_i^{(h)}(\rho)\right) \leq 2 \rho^d L_{\infty}(\rho) = 1 -
\rho^c & < 1.
\end{align*}
Second inequality holds because $K_i^{(h)}(\rho) > 0$ for all $i \geq 0$.
The last inequality holds because we have $\rho > 0$ for all $a,b,c,d$ that satisfies Assumption~\ref{i:ass:1}.
Therefore, for all $i$ from the Newton-Puiseux expansion of $L_{\infty}(z)$ and $K_i^{(h)}(z)$ at singularity $\rho$ (see Proposition~\ref{i:prop:1} and Lemma~\ref{l:2}) we have :
\begin{equation}\label{l:eq:6}
\frac{1}{1 - z^d \left(2 L_{\infty}(z) - K_i^{(h)}(z)\right)} = \dot{c}_{i} + \dot{d}_{i} \left(1-\frac{z}{\rho}\right)^{\frac12} + \BigO{1-\frac{z}{\rho}},
\end{equation}
where $\dot{c}_i = \frac{1}{1 - \rho^d \left(2 a_{\infty} - c_i^{(h)}\right)}, \dot{d}_i = \frac{\rho^d \left(2 b_{\infty} - d_i^{(h)}\right)}{\left(1 - \rho^d \left(2 a_{\infty} - c_i^{(h)}\right)\right)^2}$.
Moreover, we have
\[
\prod_{i=m}^{j} \frac{1}{1 - z^d (2 L_{\infty}(z) - K_i^{(h)}(z))} = \dot{c}_{m,j} + \dot{d}_{m,j} \left(1-\frac{z}{\rho}\right)^{\frac12} + \BigO{1-\frac{z}{\rho}},
\]
where $\dot{c}_{m,j} = \prod_{i=m}^{j} \dot{c}_i, \dot{d}_{m,j} = \sum_{i=m}^{j} \frac{\dot{d}_i}{\dot{c}_i} \prod_{k=m}^{j} \dot{c}_{k}$.
Since $\left(1-\frac{z}{\rho}\right)^{\frac12}$ does not depend on $m$ and $j$ and $\frac{z^{a - cm}}{1 - z^b}$ have poles only on the unit circle, we have for all $m \geq 0$ the following
\begin{equation*}
\frac{z^{a-cm}}{1-z^b} \sum_{j=m}^{H-1} z^{j(b+c)} \prod_{i=m}^{j} \frac{1}{1 - z^d \left(2 L_{\infty}(z) -K_i^{(h)}(z)\right)} = \ddot{c_m} + \ddot{d_m} \left(1-\frac{z}{\rho}\right)^{\frac12} + \BigO{1-\frac{z}{\rho}},
\end{equation*}
where $\ddot{c_m} = \frac{\rho^{a - c m}}{1-\rho^b} \sum_{j=m}^{H} \rho^{j(b+c)} \dot{c}_{m,j}, \ddot{d_m} = \frac{\rho^{a - c m}}{1-\rho^b} \sum_{j=m}^{H} \rho^{j(b+c)} \dot{d}_{m,j}$.

In case of the remainder: $\frac{z^{a + bH + c (H-m)}}{\left(1-z^d\right) (1-z^{b+c}-2z^d
L_{\infty}(z))} \left(\prod_{i=m}^{H-1} \frac{1}{1 - z^d \left(2 L_{\infty}(z)
-K_i^{(h)}(z)\right)} \right)$ we proceed similarly to the proof of Lemma~\ref{l:2}.
We already know how to handle the product part of this expression, so let us consider the fraction: $\frac{z^{a + b H + c (H-m)}}{\left(1-z^d\right) (1-z^{b+c}-2z^d L_{\infty}(z))}$.
Similarly to before, we have to check that the singularity of this function does not come from the root of the denominator but from the function $L_{\infty}(z)$.
So the following inequality has to hold:
\begin{align*}
1 - \rho^{b+c} - 2 \rho^{d} L_{\infty}(\rho) &> 0 \\
2 \rho^d L_{\infty}(\rho) = 1 - \rho^c & <  1 - \rho^{b+c}.
\end{align*}
Second inequality holds because we have $0 < \rho^{b+c} < \rho^{c}$ for all $a,b,c,d$ that satisfies Assumption~\ref{i:ass:1}.
Now, similarly to the previous case using the Newton-Puiseux expansion of $L_{\infty}(z)$ at $\rho$ we can derive expansion of this part of Equation~(\ref{l:eq:2}).
\end{proof}

\end{document}